\documentclass[fleqn,reqno,11pt,a4paper,final]{amsart}

\usepackage[a4paper,left=30mm,right=30mm,top=30mm,bottom=30mm,marginpar=20mm]{geometry} 
\usepackage{amsmath}
\usepackage{amssymb}
\usepackage{amsthm}
\usepackage{amscd}
\usepackage[ansinew]{inputenc}
\usepackage{cite}
\usepackage{bbm}
\usepackage{color}
\usepackage[english=american]{csquotes}
\usepackage[final]{graphicx}
\usepackage{hyperref}
\usepackage{calc}
\usepackage{mathptmx}

\graphicspath{{../Pictures/}}

\numberwithin{equation}{section}

\newtheoremstyle{thmlemcorr}{10pt}{10pt}{\itshape}{}{\bfseries}{.}{10pt}{{\thmname{#1}\thmnumber{ #2}\thmnote{ (#3)}}}
\newtheoremstyle{thmlemcorr*}{10pt}{10pt}{\itshape}{}{\bfseries}{.}\newline{{\thmname{#1}\thmnumber{ #2}\thmnote{ (#3)}}}
\newtheoremstyle{remexample}{10pt}{10pt}{}{}{\bfseries}{.}{10pt}{{\thmname{#1}\thmnumber{ #2}\thmnote{ (#3)}}}
\newtheoremstyle{ass}{10pt}{10pt}{}{}{\bfseries}{.}{10pt}{{\thmname{#1}\thmnumber{ A#2}\thmnote{ (#3)}}}

\theoremstyle{thmlemcorr}
\newtheorem{theorem}{Theorem}
\numberwithin{theorem}{section}
\newtheorem{lemma}[theorem]{Lemma}
\newtheorem{corollary}[theorem]{Corollary}

\theoremstyle{thmlemcorr*}
\newtheorem{theorem*}{Theorem}
\newtheorem{lemma*}[theorem]{Lemma}
\newtheorem{corollary*}[theorem]{Corollary}
\newtheorem{proposition*}[theorem]{Proposition}
\newtheorem{problem*}[theorem]{Problem}
\newtheorem{conjecture*}[theorem]{Conjecture}
\newtheorem{definition*}[theorem]{Definition}

\theoremstyle{remexample}
\newtheorem{remark}[theorem]{Remark}

\theoremstyle{ass}

\DeclareMathOperator{\diverg}{div}

\newcommand{\norm}[1]{\|#1\|}

\newcommand{\R}{\mathbb{R}}
\newcommand{\T}{\mathbb{T}}

\newcommand{\term}[1]{\textbf{#1}}

\def\XXint#1#2#3{{\setbox0=\hbox{$#1{#2#3}{\int}$} 
\vcenter{\hbox{$#2#3$}}\kern-.5\wd0}}

\renewcommand{\epsilon}{\varepsilon}
\renewcommand{\phi}{\varphi}

\begin{document}


\title[Weak-Strong Uniqueness]{Weak-Strong Uniqueness for Measure-Valued Solutions of Some Compressible Fluid Models}

\author{Piotr Gwiazda}
\address{\textit{Piotr Gwiazda:}  Institute of Applied Mathematics and Mechanics, University of Warsaw, Banacha 2, 02-097 Warszawa, Poland}
\email{pgwiazda@mimuw.edu.pl}

\author{Agnieszka \'{S}wierczewska-Gwiazda}
\address{\textit{Agnieszka \'{S}wierczewska-Gwiazda:} Institute of Applied Mathematics and Mechanics, University of Warsaw, Banacha 2, 02-097 Warszawa, Poland }
\email{aswiercz@mimuw.edu.pl}

\author{Emil Wiedemann}
\address{\textit{Emil Wiedemann:}Hausdorff Center for Mathematics and Mathematical Institute, Universit\"{a}t Bonn, Endenicher Allee 60, 53115 Bonn, Germany.}
\email{emil.wiedemann@hcm.uni-bonn.de}

\thanks{P. G. is a coordinator of the International Ph.D. Projects Programme of the Foundation for Polish Science operated
within the Innovative Economy Operational Programme 2007--2013
(Ph.D. Programme: Mathematical Methods in
Natural Sciences).
The research of A. \'S.-G. has received funding from the National Science Centre, DEC-2012/05/E/ST1/02218. The authors appreciate the support of the Warsaw Center of Mathematics and Computer Science}

\begin{abstract} We prove weak-strong uniqueness in the class of admissible measure-valued solutions for the isentropic Euler equations in any space dimension and for the Savage-Hutter model of granular flows in one and two space dimensions. For the latter system, we also show the complete dissipation of momentum in finite time, thus rigorously justifying an assumption that has been made in the engineering and numerical literature. 

\vspace{4pt}




\end{abstract}


\hypersetup{
  pdfauthor = {},
  pdftitle = {},
  pdfsubject = {},
  pdfkeywords = {}
}


\maketitle




\section{Introduction}
A measure-valued solution to a partial differential equation (or a system of equations) is, roughly speaking, a map that gives for every point in the domain a probability distribution of values, and that satisfies the equation only in an average sense. If this probability distribution reduces to a point mass almost everywhere in the domain, then the measure-valued solution is simply a solution in the sense of distributions. The main advantage of measure-valued solutions is the fact that, in many situations, they can easily be obtained from weakly convergent sequences of approximate solutions, even when the convergence of the approximating sequence to a distributional solution may fail due to effects of oscillation and concentration.

Measure-valued solutions to hyperbolic conservation laws were introduced by DiPerna~\cite{diperna}. He showed for scalar conservation laws in one space dimension that measure-valued solutions exist and are, under the assumption of entropy admissibility, in fact concentrated at one point, i.e.\ they can be identified with a distributional (entropy) solution. In other words, in this case the formation of fast oscillations, which corresponds to a measure with positive variance, can be excluded.

In many other physically relevant systems, however, no such compactness arguments are available, and existence of admissible weak (i.e.\ distributional) solutions seems hopeless. In such cases, the existence of measure-valued solutions is the best one can hope for.

For the incompressible Euler equations, DiPerna and Majda~\cite{dipernamajda} showed the global existence of measure-valued solutions for any initial data with finite energy. The main point of their work was to introduce so-called \emph{generalised Young measures}, which take into account not only oscillations, but also concentrations. 

Subsequently, measure-valued solutions were shown to exist for further models of fluid and gas dynamics, e.g.\ compressible Euler and Navier-Stokes equations~\cite{neustupa, kroner} or the Savage-Hutter avalanche model~\cite{gwiazda2005}.

Measure-valued solutions have been criticised for being a too weak notion of solution. Indeed, the by now fairly standard procedure of establishing measure-valued solutions by viscous approximation, thereby circumventing delicate problems of compactness, suggests that the solution thus obtained does not carry enough information to be of much use. In particular, in the absence of admissibility criteria, measure-valued solutions are obviously non-unique to a large extent, as they only contain information on certain moments of the measure.

It is therefore surprising that, in the case of the incompressible Euler equations, the so-called \emph{weak-strong uniqueness} property was proved, on the whole space, for admissible measure-valued solutions by Brenier, De Lellis, and Sz\'{e}kelyhidi~\cite{weak-strong}. This means that if there exists a sufficiently regular (classical) solution, then every admissible measure-valued solution with the same initial data will coincide with the classical solution. Admissibility means that the kinetic energy of the solution never exceeds the initial energy. 

In fact, in~\cite{lions}, P.-L. Lions required any reasonable concept of (very) weak solution to satisfy global existence and weak-strong uniqueness. For the incompressible Euler equations, therefore, admissible measure-valued solutions qualify. It is important though to emphasize the necessity of admissibility: Without this assumption, various examples are known where weak-strong uniqueness fails even for distributional solutions of incompressible Euler~\cite{scheffer, shnirel1, euler1, eulerexistence}. Also, uniqueness need not hold for admissible solutions in the absence of a strong solution, see~\cite{euler2, euleryoung, daneri}.  

We consider in this article two systems of equations in the realm of compressible fluid dynamics: The isentropic Euler equations,
\begin{equation}\label{eulerintro}
\begin{aligned} 
\partial_th+\diverg(hu)&=0\\
\partial_t(hu)+\diverg(hu\otimes u)+\nabla(\kappa h^\gamma)&=hG,
\end{aligned}
\end{equation} 
in any space dimension greater or equal one, and the Savage-Hutter equations
\begin{equation}\label{savhutintro}
\begin{aligned} 
\partial_th+\diverg(hu)&=0\\
\partial_t(hu)+\diverg(hu\otimes u)+\nabla(ah^2)&=h\left(-dB(u)+f\right),
\end{aligned}
\end{equation}
which make sense (from a modelling viewpoint) in one or two space dimensions. Here, $G$ and $f$ are external force densities, and $B(u)$ is a maximal monotone set-valued map. 
The latter system describes the evolution of the depth-averaged velocity and height of some material sliding over an inclined slope. The material is subject to the so-called Coulomb-Mohr friction law. For comprehensive studies, including derivation, numerical computations and experimental results on system~\eqref{savhutintro} and its various modifications, we refer to \cite{SaHu89, GrWiHu99, BoWe2004, BoMaPe2003, GrTaNo2003, GrCu2007, HuWaPu2005, PeBoMa2008, ZaPeTaNi2010}, among others.
We prove (cf. Theorems~\ref{Eweak-strong} and \ref{weak-strong} below):
\begin{theorem}\label{1}
Let (H,U) be a solution of~\eqref{eulerintro} or~\eqref{savhutintro} such that $H$ is Lipschitz continuous in $[0,T]\times\T^n$ and $U\in C^1([0,T]\times\T^n)$. Assume also $H\geq c>0$ for some constant $c$. Then every admissible measure-valued solution of~\eqref{eulerintro} or~\eqref{savhutintro}, respectively, with the same initial data as $(H,U)$ coincides with $(H,U)$.
\end{theorem}
Of course, the precise definitions of measure-valued solutions and admissibility will be given below. It should be mentioned that weak-strong uniqueness for admissible measure-valued solutions was proved in~\cite{weak-strong} for general hyperbolic systems of conservation laws, but this was done only for oscillation measures.

Moreover, the results in~\cite{weak-strong} are valid even for Lipschitz continuous strong solutions. Owing to commutator estimates analogous to the ones provided by Constantin, E and Titi for the incompressible Euler system in~\cite{constetiti}, the result of Theorem~\ref{1} can be obtained assuming only Lipschitz continuity of $U$, and Sobolev regularity of $H$. In fact, as in~\cite{weak-strong} it is sufficient to assume only that the symmetric part of $\nabla U$ be bounded. We omit details.

Weak-strong uniqueness for compressible Euler models appears important in the light of several recent examples of non-uniqueness of admissible weak solutions~\cite{euler2, chiodaroli, chiodarolikreml, chiodarolikremlfeireisl, chiodarolidelelliskreml, feireisl}. For the Savage-Hutter equations, such examples were very recently constructed in~\cite{feireislgwiazdaswierczewska}.

To prove the weak-strong uniqueness, we follow the general strategy of~\cite{feireisljinnovotny} (where weak-strong uniqueness was proved for the compressible Navier-Stokes equations), but we have to adjust these arguments to the measure-valued framework. For the Savage-Hutter system, an additional issue is to give a meaningful definition of measure-valued solutions that accounts for the multi-valued nature of the fricition term $B(u)$ in~\eqref{savhutintro}. Such a definition was proposed in~\cite{gwiazda2005} and we will use it here as well.

If the force $f$ is time-independent and $\norm{f}_\infty<d$, then a special class of solutions to~\eqref{savhutintro} is given by $u=0$ and $h$ independent of time and such that
\begin{equation*}
\left|\nabla h(x)-\frac{f(x)}{2a}\right|\leq\frac{d}{2a}\hspace{0.2cm}\text{for almost every $x$.}
\end{equation*}
Observe that our weak-strong uniqueness result allows to take $(H,U)$ as such a stationary solution, so that in particular every such solution enjoys uniqueness in the class of admissible measure-valued solutions.

For the Savage-Hutter model we also prove the following result:
\begin{theorem}\label{finiteintro}
There exists a finite time $0\leq T<\infty$, only depending on the parameters in~\eqref{savhutintro} and the initial data, such that every admissible measure-valued solution of~\eqref{savhutintro} starting from such initial data has zero momentum for almost every time $t>T$. 
\end{theorem}
In particular, this implies that every admissible weak (distributional) solution becomes stationary after finite time. This highlights the importance of stationary solutions as well as the role played by the admissibility condition: Indeed, in~\cite{feireislgwiazdaswierczewska} non-admissible weak solutions were constructed whose momentum does not decrease to zero.  The result is a rigorous justification of empirical and numerical observations  of deposition of material after finite time,~\cite{SaHu91, Fe-Nietal2008, CoGuMo2012}. The finite-time  runout of solutions is essentially used at the modelling stage as  providing data for calibration of the system. This property was assumed in numerical simulations, however, to our knowledge, never proved. 

Let us remark that for the one-dimensional Savage-Hutter model, we obtain a fairly complete picture: Existence of admissible global in time weak solutions is known~\cite{Gw2002}, they enjoy weak-strong uniqueness, and become stationary after finite time.

Similarly, for the compressible Euler system in the one-dimensional case there exist global in time admissible weak solutions having the weak-strong uniqueness property, see~\cite{DiPe83, LiPeSo96, LiPeTa94}.

Finally, let us point out some difficulties in extending our results to other domains than the torus. On the whole space, we can no longer require the denstity $H$ to be uniformly bounded away from zero and the initial energy
\begin{equation*}
\int_{\R^n}\frac{1}{2}h_0|u_0|^2+ah_0^2dx
\end{equation*} 
(and similarly for~\eqref{eulerintro}) to be finite at the same time. On domains with physical boundaries, however, we do not even expect weak-strong uniqueness to hold, since a counterexample has been exhibited in~\cite{eulerboundary} in the case of the incompressible Euler equations.

\section{Notation}
We fix here some notation that will be used throughout the paper.

The $n$-dimensional torus will be denoted by $\T^n:=\R^n\slash \mathbb{Z}^n$.

Let $\Omega\subset\R^n$ be a measurable subset or $\Omega=\T^n$. The set of locally finite nonnegative measures on $\Omega$ will be denoted $\mathcal{M}^+(\Omega)$. If $X$ is a measurable subset of $\R^m$, then $\mathcal{P}(X)$ will be the set of probability measures on $X$. 

Let $m\in\mathcal{M}^+(\Omega)$. The space $L^\infty_w(\Omega,m;\mathcal{P}(X))$ is then defined as the space of maps $\nu:\Omega\to\mathcal{P}(X)$, $x\mapsto\nu_x$, which are \term{weakly* measurable} with respect to $m$; that is, for every $\phi\in C_c(X)$ the map 
\begin{equation*}
x\mapsto\int_{\Omega}\phi(\lambda)d\nu_x(\lambda)
\end{equation*}
is $m$-measurable. If $m$ is Lebesgue measure, we simply write $L^\infty_w(\Omega;\mathcal{P}(X))$. If $\Omega$ has the form $[0,T]\times\tilde{\Omega}$ for some measurable subset $\tilde{\Omega}\subset \R^n$, then $dx$ denotes $n$-dimensional Lebesgue measure and $dt$ one-dimensional Lebesgue measure. The Dirac mass centred at $x$ will be denoted as $\delta_x$, as usual.

The $m$-dimensional unit sphere is written $\mathbb{S}^m$. With $\bar{\Omega}$ we mean the topological closure of a subset of $\R^n$. We write $\R^+$ for the set of non-negative real numbers.

In the case $\Omega=[0,T]\times\tilde{\Omega}$, we will use measures of the form $m=m_t\otimes dt$; this means that, for every set of the form $\tau\times U$, where $\tau\subset[0,T]$ and $U\subset\tilde{\Omega}$ are measurable subsets,
\begin{equation*}
m(\tau\times U)=\int_\tau m_t(U)dt.
\end{equation*}

The differential operators $\nabla$ and $\diverg$ are applied only to the spatial variables. If $u$ and $v$ are vectors, then $u\otimes v$ denotes the matrix defined by $(u\otimes v)_{ij}=u_iv_j$. The divergence of a matrix field is understood to be taken row-wise.

Further notation will be introduced as we proceed.

\section{Generalised Young Measures}\label{young}
We recall briefly the notion of generalised Young measures, which were introduced by DiPerna and Majda~\cite{dipernamajda} and refined by Alibert and Bouchitt\'{e}~\cite{alibert}. Further details can be found e.g.\ in~\cite{k-r2, euleryoung}.

Young measures are used to represent weak limits of nonlinear functions of weakly convergent sequences. More precisely, suppose $\Omega\subset \R^n$  is a measurable set or $\Omega=\T^n$ ($n\geq1$), and $(u_k)_{n\in\mathbb{N}}$ is a sequence of maps bounded in $L^1(\Omega;\R^m)$ ($m\geq1$).

Then it was proved in~\cite{alibert} that there exists a subsequence (not relabeled) as well as a parametrised probability measure $\nu\in L_w^\infty(\Omega;\mathcal{P}(\R^m))$ (which is identical with the "classical" Young measure), a non-negative measure $m\in\mathcal{M}^+(\bar{\Omega})$, and a parametrized probability measure $\nu^\infty\in L_w^\infty(\Omega,m;\mathcal{P}(\mathbb{S}^{m-1}))$ such that
\begin{equation*}
f(x,u_n(x))dx\stackrel{*}{\rightharpoonup}\int_{\R^m}f(x,\lambda)d\nu_x(\lambda)dx+\int_{\mathbb{S}^{m-1}}f^\infty(x,\beta)d\nu^\infty_x(\beta)m
\end{equation*}
weakly* in the sense of measures. Here, $f:\Omega\times\R^m\to\R$ is any Carath\'{e}odory function (measurable in the first and continuous in the second argument) whose \term{recession function}
\begin{equation*}
f^\infty(x,\beta):=\lim_{{x'\rightarrow x\atop{\beta'\rightarrow
\beta\atop{s\rightarrow\infty}}}}\frac{f(x',s\beta')}{s}
\end{equation*} 
is a well-defined and continuous function on $\bar{\Omega}\times\mathbb{S}^{m-1}$. Note that such an $f$ will have at most linear growth. If its growth is sublinear, then $f^\infty=0$.

 Notice also that $\nu^\infty_{t,x}$ is only defined $m$-almost everywhere.

In order to properly define measure-valued solutions to compressible fluid equations within the framework of Alibert--Bouchitt\'{e}, we need a slight refinement which allows us to treat sequences whose components have different growth. Let $(u_k,w_k)_k$ be a sequence such that $(u_k)$ is bounded in $L^p(\Omega;\R^l)$ and $(w_k)$ is bounded in $L^q(\Omega;\R^m)$ ($1\leq p,q<\infty$). Define the ``nonhomogeneous unit sphere'' 
\begin{equation*}
\mathbb{S}^{l+m-1}_{p,q}:=\{(\beta_1,\beta_2)\in\R^{l+m}: |\beta_1|^{2p}+|\beta_2|^{2q}=1\}.
\end{equation*}

Then, there exists a a subsequence (not relabeled) and measures $\nu\in L_w^\infty(\Omega;\mathcal{P}(\R^{l+m}))$, $m\in\mathcal{M}^+(\bar{\Omega})$, $\nu^\infty\in L_w^\infty(\Omega,m;\mathcal{P}(\mathbb{S}^{l+m-1}_{p,q}))$ such that 
\begin{equation*}
f(x,u_n(x),w_n(x))dx\stackrel{*}{\rightharpoonup}\int_{\R^{l+m}}f(x,\lambda_1,\lambda_2)d\nu_x(\lambda_1,\lambda_2)dx+\int_{\mathbb{S}^{l+m-1}_{p,q}}f^\infty(x,\beta_1,\beta_2)d\nu^\infty_x(\beta_1,\beta_2)m
\end{equation*}
in the sense of measures; this is valid for all integrands $f$ whose $p$-$q$-recession function exists and is continuous on $\bar{\Omega}\times\mathbb{S}^{l+m-1}_{p,q}$. The $p$-$q$-recession function is defined as
\begin{equation*}
f^\infty(x,\beta_1,\beta_2):=\lim_{{x'\rightarrow x\atop{(\beta_1',\beta_2')\rightarrow
(\beta_1,\beta_2)\atop{s\rightarrow\infty}}}}\frac{f(x',s^q\beta_1',s^p\beta_2')}{s^{pq}}.
\end{equation*} 
The case $p=2$, $q=1$ was treated in Subsection 2.4.1 of~\cite{euleryoung} and the extension to general $p$ and $q$ is straightforward.

Let us quote another fact which is important for measure-valued solutions of time-dependent equations with bounded energy: If $\Omega=[0,T]\times\tilde{\Omega}$ for some measurable $\tilde{\Omega}\subset\R^n$ (or $\tilde{\Omega}=\T^n$), and if the sequence $(u_n,w_n)_n$ is bounded in $L^\infty([0,T];L^p(\tilde{\Omega})\times L^q(\tilde{\Omega}))$, then the corresponding concentration measure $m$ admits a disintegration of the form
\begin{equation*}
m=m_t(dx)\otimes dt,
\end{equation*}
where $t\mapsto m_t$ is bounded and measurable viewed as a map from $[0,T]$ into $\mathcal{M}^+(\overline{\tilde{\Omega}})$. The proof of this statement was given in~\cite{weak-strong}.

\section{Weak-strong uniqueness for measure-valued solutions of the compressible Euler equations}\label{Euler}
We consider the compressible Euler system
\begin{equation}\label{euler}
\begin{aligned} 
\partial_th+\diverg(hu)&=0\\
\partial_t(hu)+\diverg(hu\otimes u)+\nabla(\kappa h^\gamma)&=hG.
\end{aligned}
\end{equation}
Here, $h:[0,T]\times\T^n\to\R$, $u:[0,T]\times\T^n\to\R^n$, and $G:[0,T]\times\T^n\to\R^n$, $\gamma>1$. We set the constant $\kappa>0$ equal to one in order to save some writing, remarking however that all computations remain unchanged for general $\kappa$.

\subsection{Measure-valued solutions}
We apply the abstract framework from the previous section, with $l=1$, $m=n$, $p=\gamma$, and $q=2$, in order to define the notion of \term{measure-valued solution} of~\eqref{euler}. Consider a generalised Young measure 
\begin{equation*}
(\nu_{t,x},m,\nu^\infty_{t,x})\in L_w^\infty\left([0,T]\times\T^n;\mathcal{P}(\R^+\times\R^n)\right)\times\mathcal{M}^+([0,T]\times\T^n)\times L_w^\infty\left([0,T]\times\T^n,m;\mathcal{P}(\mathbb{S}^+)\right),
\end{equation*}
where we wrote
\begin{equation*}
\mathbb{S^+}:=\{(\beta_1,\beta')\in\mathbb{S}^{1+n}_{\gamma,2}: \beta_1\geq0\}.
\end{equation*}
We will use the variables $(\lambda_1,\lambda')\in\R^+\times\R^n$ and $(\beta_1,\beta')\in\mathbb{S}^+$ as dummy variables when integrating with respect to $\nu_{t,x}$ and $\nu_{t,x}^\infty$, respectively. One should think of $\lambda_1,\beta_1$ as representing $h$ and $\lambda', \beta'$ as representing $\sqrt{h}u$. We also use the common notation
\begin{equation*}
\langle F(\lambda_1,\lambda'),\nu_{t,x}\rangle:=\int_{\R^+\times\R^n}F(\lambda_1,\lambda')d\nu_{x,t}(\lambda_1,\lambda')
\end{equation*}
and analogously for $\nu^\infty$. 

If we consider a function $f:[0,T]\times\T^n\times\R^+\times\R^n\to\R$ which has an appropriate $\gamma$-2-recession function as defined in Section~\ref{young}, we use the shorthand notation
\begin{equation*}
\bar{f}(dtdx):=\langle f(t,x,\cdot,\cdot),\nu_{t,x}\rangle dtdx+\langle f^\infty(t,x,\cdot,\cdot),\nu_{t,x}^\infty\rangle m(dtdx).
\end{equation*}
For instance, we have
\begin{equation*}
\begin{aligned}
\bar{h}&= \langle\lambda_1,\nu\rangle\\
\overline{h^\gamma}&=\langle\lambda_1^\gamma,\nu\rangle+\langle\beta_1^\gamma,\nu^\infty\rangle m\\
\overline{hu}&=\langle\sqrt{\lambda_1}\lambda',\nu\rangle\\
\overline{hu\otimes u}&=\langle\lambda'\otimes\lambda',\nu\rangle+\langle\beta'\otimes\beta',\nu^\infty\rangle m\\
\overline{h|u|^2}&=\langle|\lambda'|^2,\nu\rangle+\langle|\beta'|^2,\nu^\infty\rangle m\\
\overline{hG}&=\langle\lambda_1G,\nu\rangle=\bar{h}G.
\end{aligned}
\end{equation*}
We say that $(\nu,m,\nu^\infty)$ is a \term{measure-valued solution} of~\eqref{euler} with initial data $(h_0,u_0)$ if for every $\tau\in[0,T]$, $\psi\in C^1([0,T]\times\T^n;\R)$, $\phi\in C^1([0,T]\times\T^n;\R^n)$ it holds that
\begin{equation}\label{Emass_momentum}
\begin{aligned}
\int_0^\tau\int_{\T^n}\partial_t\psi \bar{h}+\nabla\psi\cdot\overline{hu}dxdt+\int_{\T^n}\psi(x,0)h_0-\psi(x,\tau)\bar{h}(x,\tau)dx&=0,\\
\int_0^\tau\int_{\T^n}\partial_t\phi\cdot\overline{hu}+\nabla\phi : \overline{hu\otimes u}+\diverg\phi\overline{h^\gamma}
-\phi\cdot\overline{hG}&dxdt\\
+\int_{\T^n}\phi(x,0)\cdot h_0u_0-\phi(x,\tau)\cdot\overline{hu}(x,\tau)dx&=0.
\end{aligned}
\end{equation}
It is part of the definition that all the integrals have to exist for any choice of test functions, in particular for the initial data we require $h_0\in L^1$, $h_0u_0\in L^1$.

Let us set
\begin{equation*}
E_{mvs}(t):=\int_{\T^n}\frac{1}{2}\overline{h|u|^2}(t,x)+\frac{1}{\gamma-1}\overline{h^\gamma}(t,x)dx
\end{equation*}
for almost every $t$, and
\begin{equation*}
E_0:=\int_{\T^n}\frac{1}{2}h_0|u_0|^2(x)+\frac{1}{\gamma-1}h_0^\gamma(x)dx.
\end{equation*}
We then say that a measure-valued solution is \term{admissible} if
\begin{equation}\label{EEmvsenergy}
\begin{aligned}
E_{mvs}(t)\leq E_0+\int_0^t\int_{\T^n}\overline{hG\cdot u}(s,x)dxds
\end{aligned}
\end{equation}
in the sense of distributions. An elementary computation yields the well-known fact that the energy is conserved for smooth solutions (i.e.~\eqref{EEmvsenergy} holds with equality), whereas the inequality becomes strict upon the formation of shocks.

\begin{remark}
The global existence of measure-valued solutions for \eqref{euler} was proved by Neustupa in \cite{neustupa}. However he used a different formulation of the Young measure, as the formalism of Alibert-Bouchitt\'{e} \cite{alibert} was not yet available. One can however rewrite the solutions of \cite{neustupa} in the form presented here. Neustupa's solutions can be seen to be admissible, as they can be obtained e.g.\ from an artificial viscosity approximation.
\end{remark}

\subsection{Weak-Strong Uniqueness}
\begin{theorem}\label{Eweak-strong}
Let $G\in L^{\infty}([0,T];L^2(\T^n))$ and suppose $H\in W^{1,\infty}([0,T]\times\T^n), U\in C^1([0,T]\times\T^n)$ is a solution of~\eqref{euler} with initial data $h_0\geq c>0$, $h_0\in L^\gamma(\T^n)$, $h_0|u_0|^2\in L^1(\T^n)$, and $H(x,t)\geq c>0$ for some constant $c$ and all $(t,x)\in[0,T]\times\T^n$. If $(\nu,m,\nu^\infty)$ is an admissible measure-valued solution with the same initial data, then 
\begin{equation*}
\nu_{t,x}=\delta_{(H(t,x),\sqrt{H(t,x)}U(t,x))} \text{ for a.e. $t,x$, and $m=0$.}
\end{equation*} 
\end{theorem}
\begin{proof}
Let us first define for a.e.\ $t\in[0,T]$ the \term{relative energy} between $(H,U)$ and the measure-valued solution as
\begin{equation*}
\begin{aligned}
E_{rel}(t)&=\int_{\T^n}\frac{1}{2}\overline{h|u-U|^2}+\overline{\frac{1}{\gamma-1}h^\gamma-\frac{\gamma}{\gamma-1}H^{\gamma-1}h+H^\gamma}dx\\
&=\int_{\T^n}\frac{1}{2}\langle|\lambda'-\sqrt{\lambda_1}U|^2,\nu_{t,x}\rangle dx+\frac{1}{2}\int_{\T^n}\langle|\beta'|^2,\nu_{t,x}^\infty\rangle dm_t(x)\\
&\hspace{1cm}+\int_{\T^n}\langle\frac{1}{\gamma-1}\lambda_1^\gamma-\frac{\gamma}{\gamma-1}H^{\gamma-1}\lambda_1+H^\gamma,\nu_{t,x}\rangle dx\\
&\hspace{1cm}+\int_{\T^n} \frac{1}{\gamma-1}\langle\beta_1^\gamma,\nu^\infty_{t,x}\rangle dm_t(x).
\end{aligned}
\end{equation*}
Here, the measure $m_t\in\mathcal{M}^+(\T^n)$ is obtained by the disintegration $m(dtdx)=m_t(dx)\otimes dt$, which is well-defined thanks to the admissibility (cf.\ Section~\ref{young}). Note that the strict convexity of $|\cdot|^\gamma$ implies that the relative energy is always non-negative. Then it is straightforward to observe that $E_{rel}(t)=0$ for a.e.\ $t$ implies Theorem~\ref{Eweak-strong}. Indeed, defining the projection operators $\pi^{\lambda_1}:(\lambda_1, \lambda')\mapsto \lambda_1$ and $\pi^{\lambda'}:(\lambda_1, \lambda')\mapsto \lambda'$, we observe that the strict convexity of  $|\cdot|^\gamma$ implies that $\pi^{\lambda_1}\nu_{t,x}=\delta_{H(t,x)}$ for a.e. $t,x$ and hence $\nu_{t,x}=\delta_{(H(t,x))}\otimes \pi^{\lambda'}\nu_{t,x}$. Using the first term in the relative energy allows to conclude $\pi^{\lambda'}\nu_{t,x}=\delta_{(\sqrt{H(t,x)}U(t,x))}$.

Setting $\phi=U$ in the momentum equation (the second equation of~\eqref{Emass_momentum}), we obtain 
\begin{equation}\label{Etest1}
\begin{aligned}
\int_{\T^n}\overline{hu}\cdot U(\tau)dx=&\int_{\T^n}h_0|u_0|^2dx+\int_0^\tau\int_{\T^n}\overline{hu}\cdot\partial_t U+\overline{hu\otimes u}:\nabla U dxdt\\
&+\int_0^\tau\int_{\T^n}\overline{h^\gamma}\diverg U+\overline{hG}\cdot U dxdt.
\end{aligned}
\end{equation}
Similarly, setting $\psi=\frac{1}{2}|U|^2$ and then $\psi=\gamma H^{\gamma-1}$ in~\eqref{Emass_momentum} yields
\begin{equation}\label{Etest2}
\frac{1}{2}\int_{\T^n}|U(\tau)|^2\bar{h}(\tau,x)dx=\int_0^\tau\int_{\T^n}U\cdot\partial_tU \bar{h}+\nabla U U\cdot\overline{hu}dxdt+\int_{\T^n}\frac{1}{2}|u_0|^2h_0dx
\end{equation}
and
\begin{equation}\label{Etest3}
\int_{\T^n}\gamma H^{\gamma-1}(\tau)\bar{h}(\tau)dx=\int_0^\tau\int_{\T^n}\gamma(\gamma-1)H^{\gamma-2}\partial_tH \bar{h}+\gamma(\gamma-1)H^{\gamma-2}\nabla H \cdot\overline{hu}dxdt+\int_{\T^n}\gamma h_0^\gamma dx,
\end{equation}
respectively.

Next, we can write the relative energy as
\begin{equation*}
\begin{aligned}
E_{rel}(\tau)&=\int_{\T^n}\frac{1}{2}\overline{h|u|^2}+\frac{1}{\gamma-1}\overline{h^\gamma}dx + \int_{\T^n}H^\gamma dx+\frac{1}{2}\int_{\T^n}|U|^2\bar{h}dx-\int_{\T^n}U\cdot\overline{hu}dx-\int_{\T^n}\frac{\gamma}{\gamma-1}H^{\gamma-1}\bar{h}dx\\
&=E_{mvs}(\tau)+ \int_{\T^n}H^\gamma dx+\frac{1}{2}\int_{\T^n}|U|^2\bar{h}dx-\int_{\T^n}U\cdot\overline{hu}dx-\int_{\T^n}\frac{\gamma}{\gamma-1}H^{\gamma-1}\bar{h}dx
\end{aligned}
\end{equation*}
(all integrands evaluated at time $\tau$). Next, using the balances~\eqref{Etest1},~\eqref{Etest2},~\eqref{Etest3} for the last three integrals, we obtain
\begin{equation*}
\begin{aligned}
E_{rel}(\tau)=E_{mvs}(\tau)&+\int_{\T^n}H^\gamma dx\\
&+\int_0^\tau\int_{\T^n}U\cdot\partial_tU \bar{h}+\nabla U U\cdot\overline{hu}dxdt+\int_{\T^n}\frac{1}{2}|u_0|^2h_0dx\\ 
&-\int_{\T^n}h_0|u_0|^2dx-\int_0^\tau\int_{\T^n}\overline{hu}\cdot\partial_t U+\overline{hu\otimes u}:\nabla U dxdt\\
&-\int_0^\tau\int_{\T^n}\overline{h^\gamma}\diverg U-\overline{hG}\cdot U dxdt\\
&-\int_0^\tau\int_{\T^n}(\gamma H^{\gamma-2}\partial_tH \bar{h}+\gamma H^{\gamma-2}\nabla H \cdot\overline{hu})dxdt-\int_{\T^n}\frac{\gamma}{\gamma-1} h_0^\gamma dx,
\end{aligned}
\end{equation*}
and using~\eqref{Emvsenergy} we have, for a.e.\ $\tau$,
\begin{equation}\label{Eintermediatestep}
\begin{aligned}
E_{rel}(\tau)\leq& -\int_{\T^n}h_0^\gamma dx +\int_0^\tau\int_{\T^n}\overline{hG\cdot u} +\int_{\T^n}H^\gamma dx\\
&+\int_0^\tau\int_{\T^n}U\cdot\partial_tU \bar{h}+\nabla U U\cdot\overline{hu}dxdt\\ 
&-\int_0^\tau\int_{\T^n}\overline{hu}\cdot\partial_t U+\overline{hu\otimes u}:\nabla U dxdt\\
&-\int_0^\tau\int_{\T^n}\overline{h^\gamma}\diverg U-\overline{hG}\cdot U dxdt\\
&-\int_0^\tau\int_{\T^n}(\gamma H^{\gamma-2}\partial_tH \bar{h}+\gamma H^{\gamma-2}\nabla H \cdot\overline{hu})dxdt.
\end{aligned}
\end{equation}
Next, we collect some terms and write
\begin{equation}\label{Eequality1}
\begin{aligned}
\int_{\T^n}&H^\gamma dx -\int_{\T^n}h_0^\gamma dx-\int_0^\tau\int_{\T^n}\gamma H^{\gamma-2}\partial_tH \bar{h}dxdt\\
&=\int_0^\tau\int_{\T^n}\frac{d}{dt}H^\gamma-\gamma H^{\gamma-2}\partial_tH \bar{h}dxdt\\
&=\int_0^\tau\int_{\T^n}\gamma H^{\gamma-1}\partial_tH-\gamma H^{\gamma-2}\partial_tH \bar{h}dxdt\\
&=\int_0^\tau\int_{\T^n}\gamma H^{\gamma-2}\partial_tH(H-\bar{h})dxdt,
\end{aligned}
\end{equation}

\begin{equation}\label{Eequality2}
\begin{aligned}
\int_0^\tau\int_{\T^n}&\overline{hG\cdot u}-\overline{hG}\cdot U dxdt\\
&=\int_0^\tau\int_{\T^n}\overline{hG\cdot (u-U)}dxdt,
\end{aligned}
\end{equation}
and
\begin{equation}\label{Eequality3}
\begin{aligned}
\int_0^\tau\int_{\T^n}&U\cdot\partial_tU \bar{h}+\nabla U U\cdot\overline{hu}-\overline{hu}\cdot\partial_t U-\overline{hu\otimes u}:\nabla U dxdt\\
&=\int_0^\tau\int_{\T^n}\partial_tU\cdot \overline{h(U-u)}+\nabla U :\overline{hu\otimes(U-u)}dxdt.
\end{aligned}
\end{equation}
Indeed, the last two equalities can be verified by writing the expressions in the "coarse-grained" overline notation explicitly in terms of the Young measure $(\nu,m,\nu^\infty)$. 

Plugging equalities~\eqref{Eequality1},~\eqref{Eequality2},~\eqref{Eequality3} into~\eqref{Eintermediatestep}, we arrive at
\begin{equation}\label{Eintermediatestep2}
\begin{aligned}
E_{rel}(\tau)\leq&\int_0^\tau\int_{\T^n}\gamma H^{\gamma-2}\partial_tH(H-\bar{h})dxdt \\
&+\int_0^\tau\int_{\T^n}\overline{hG\cdot (u-U)}dxdt\\
&+\int_0^\tau\int_{\T^n}\partial_tU\cdot \overline{h(U-u)}+\nabla U :\overline{hu\otimes(U-u)}dxdt\\
&-\int_0^\tau\int_{\T^n}\overline{h^\gamma}\diverg U dxdt-\int_0^\tau\int_{\T^n}\gamma H^{\gamma-2}\nabla H \cdot\overline{hu}dxdt.
\end{aligned}
\end{equation}

For the last two integrals, we have by the divergence theorem
\begin{equation}\label{Eh^2estimate}
\begin{aligned}
-\int_0^\tau\int_{\T^n}&\overline{h^\gamma}\diverg U dxdt-\int_0^\tau\int_{\T^n}\gamma H^{\gamma-2}\nabla H \cdot\overline{hu}dxdt\\
&=\int_0^\tau\int_{\T^n}-\overline{h^\gamma}\diverg U +\gamma H^{\gamma-2}\nabla H \cdot(HU-\overline{hu})-\gamma H^{\gamma-2}\nabla H\cdot HUdxdt\\
&=\int_0^\tau\int_{\T^n}(H^\gamma-\overline{h^\gamma})\diverg U +\gamma H^{\gamma-2}\nabla H \cdot(HU-\overline{hu})dxdt.\\
\end{aligned}
\end{equation}
Inserting this back into~\eqref{Eintermediatestep2} and observing that, by the mass equation for $(H,U)$, 
\begin{equation*}
\begin{aligned}
\gamma &H^{\gamma-2}\partial_tH(H-\bar{h})+\gamma H^{\gamma-2}\diverg UH(H-\bar{h}) +\gamma H^{\gamma-2}\nabla H \cdot HU
=\gamma H^{\gamma-2}U\cdot\nabla H\bar{h},
\end{aligned}
\end{equation*}
we get
\begin{equation}\label{Eintermediatestep3}
\begin{aligned}
E_{rel}(\tau)\leq&\int_0^\tau\int_{\T^n}\gamma H^{\gamma-2}\cdot\nabla H\overline{h(U-u)}dxdt \\
&+\int_0^\tau\int_{\T^n}\overline{hG\cdot (u-U)}dxdt\\
&+\int_0^\tau\int_{\T^n}\partial_tU\cdot \overline{h(U-u)}+\nabla U :\overline{hu\otimes(U-u)}dxdt\\
&-\int_0^\tau\int_{\T^n}\gamma H^{\gamma-1}\diverg U(H-\overline{h} )dxdt+\int_0^\tau\int_{\T^n}(H^\gamma-\overline{h^\gamma})\diverg U.
\end{aligned}
\end{equation}
The expression in the third line can be rewritten pointwise as 
\begin{equation}\label{Epointwiseest}
\begin{aligned}
\partial_tU&\cdot \overline{h(U-u)}+\nabla U :\overline{hu\otimes(U-u)}\\
&=\partial_tU\cdot \overline{h(U-u)}+\nabla U :\overline{hU\otimes(U-u)}+\nabla U :\overline{h(u-U)\otimes(U-u)},
\end{aligned}
\end{equation}
and the integral of the last term as well as the last line in~\eqref{Eintermediatestep3} can both be estimated by
\begin{equation}\label{Eabsorb}
C\norm{U}_{C^1}\int_0^\tau E_{rel}(t)dt.
\end{equation}
For the remaining terms in~\eqref{Epointwiseest} we obtain, using the momentum equation for $(H,U)$,
\begin{equation}\label{Emomentum}
\begin{aligned}
\partial_tU&\cdot \overline{h(U-u)}+\nabla U :U\otimes\overline{h(U-u)}\\
&=\frac{1}{H}(\partial_t(HU)+\diverg(HU\otimes U))\cdot\overline{h(U-u)}\\
&=G\cdot\overline{h(U-u)} - \gamma H^{\gamma-2}\nabla H\cdot\overline{h(U-u)}.
\end{aligned}
\end{equation}
Putting together~\eqref{Eintermediatestep3},~\eqref{Eabsorb}, and~\eqref{Emomentum}, we obtain
\begin{equation*}
E_{rel}(\tau)\leq C\norm{U}_{C^1}\int_0^\tau E_{rel}(t)dt.
\end{equation*}
Finally,  from Gronwall's inequality it follows that $E_{rel}(\tau)=0$ for a.e. $t$.
\end{proof}

\section{Savage-Hutter system} 
We consider the two-dimensional Savage-Hutter model
\begin{equation}\label{savhut}
\begin{aligned} 
\partial_th+\diverg(hu)&=0\\
\partial_t(hu)+\diverg(hu\otimes u)+\nabla(ah^2)&=h\left(-dB(u)+f\right).
\end{aligned}
\end{equation}
The one-dimensional case can be treated similarly. Here, $h:[0,T]\times\T^2\to\R$, $u:[0,T]\times\T^2\to\R^2$, $f:[0,T]\times\T^2\to\R^2$, and $a>0$ and $d>0$ are constant. By $B(u)$ we denote the subdifferential of $u\mapsto |u|$, so that $B(u)$ is multi-valued such that
\begin{equation*}
B(u)=\begin{cases}
\frac{u}{|u|} & \text{if $u\neq0$,}\\[0.2cm]
\overline{B_1(0)} & \text{if $u=0$.}
\end{cases}
\end{equation*}
Consequently, the equality sign in the second line of~\eqref{savhut} should really be an inclusion. We will stick however to the formulation~\eqref{savhut}, thereby slightly abusing notation.

\subsection{Stationary solutions}If in~\eqref{savhut} $f$ is independent of time, then a special class of solutions is given by $u\equiv0$ and any $h=h(x)>c$ such that
\begin{equation*}
\left|\nabla h-\frac{f}{2a}\right|\leq\frac{d}{2a}\hspace{0.3cm}\text{for a.e.\ $x$.}
\end{equation*}

\subsection{Measure-Valued Solutions}
We recall  the notion of \term{measure-valued solution} of~\eqref{savhut} from \cite{gwiazda2005} in the notation used therein (in fact, there the problem was treated on the whole space, but it can easily be adapted to the torus). The author considers system~\eqref{savhut} with a right-hand side given by $h\tilde f(x, \sqrt{h} u)$, where
\begin{equation*}
\tilde f(x,\sqrt{h} u)=-d\frac{\sqrt h u}{\sqrt h |u|}+f
\end{equation*}
and for $u=0$ the mapping $\tilde f$ takes values in the closed unit ball. 
 To handle this multi-valued (monotone) term, let us first recall from~\cite{gwiazda2005} the following observation, see also \cite{GwZa2007, BuGwMaSw2009} for a similar approach.

\begin{lemma}\label{monotoniczna-ciaglosc}
Let $f:\R^n \rightarrow \R^n$ ($f:\R^n \rightarrow 2^{\R^n}$)  be a monotone
function (monotone mapping).
Then $$(f+Id)^{-1}:\R^n\rightarrow \R^n $$
and
$$f\circ (f+Id)^{-1}:\R^n\rightarrow \R^n $$
are Lipschitz functions. Above we understand  $f\circ(f+Id)^{-1}=Id-Id\circ(f+Id)^{-1}.$

Moreover, for any continuous function $g:\R^n \rightarrow \R^n$,
$$g\circ (f+Id)^{-1}:\R^n\rightarrow \R^n$$
is a continuous function.
\end{lemma}

Under the assumptions that 
 $(h_0,h_0u_0) \in L^1_{loc}(\R^2),$
$\int_{\R^2}\left\{\frac{1}{2} |u_0|^2h_0+a(h_0)^2 \right\}dx<\infty$,
there exists a triple of measures
\begin{equation*}
(\mu_{t,x},m,\mu^\infty_{t,x})\in L_w^\infty\left([0,T]\times\T^2;\mathcal{P}(\R^+\times\R^2)\right)\times\mathcal{M}^+([0,T]\times\T^2)\times L_w^\infty\left([0,T]\times\T^2,m;\mathcal{P}(\mathbb{S}^+)\right),
\end{equation*}
%
%
such that
\begin{equation*}
\begin{aligned}
\int_{[0,T)}\int_{\R^2}\overline{h}\partial_t \varphi_1
+\overline{m}\cdot \nabla_x
\varphi_1 dxdt&=\int_{\R^2}h_0\varphi_1(0)dx,\\
\int_{[0,T)} \left\{ \int_{\R^2}\overline{m}\partial_t \varphi_2 dx+
\langle \overline{e}+a \overline{p},\nabla_x\varphi_2 \rangle -
\int_{\R^2}\overline{f}\varphi_2 dx \right\} dt
&=\int_{\R^2}h_0u_0\varphi_2(0)dx
\end{aligned}
\label{MV-SH}
\end{equation*}
for all $\varphi_1, \varphi_2 \in C_c([0,T)\times\R^2)$, and
for almost all $t\in[0,T)$ it holds that
\begin{equation}\label{MV-NE}
\langle {\rm Tr}(\overline{e}(t))+a\overline{p}(t), 1\rangle
-\int_{\R^2}\left\{\frac{1}{2} |u_0|^2h_0+a(h_0)^2 \right\}dx
\leq\int_{[0,T)\times\R^2}\chi dxdt,
\end{equation}
where
\begin{equation*}
\begin{split}
\overline{h}(t,x)=&\int_{\R_+\times\R^2}
\lambda_1\,\,d\mu_{t,x}(\lambda),\\
\overline{p}(t,x)=&\int_{\R_+\times\R^2}
\lambda_1^2\,\,d\mu_{t,x}(\lambda)
+\Big(\int_{S^2_+}\beta_1^2\,\,d\nu_{t,x}^\infty(\beta) \Big) m,\\
\overline{m}(t,x)=&\int_{\R_+\times\R^2}
\sqrt{\lambda_1}\,(-\tilde{f}+Id)^{-1}
(x,(\lambda_2,\lambda_3))\,\,d\mu_{t,x}(\lambda),\\
\overline{e}(t,x)=&\int_{\R_+\times\R^2}
(-\tilde{f}+Id)^{-1}
(x,(\lambda_2,\lambda_3))\otimes(-\tilde{f}+Id)^{-1}
(x,(\lambda_2,\lambda_3))\,\,d\mu_{t,x}(\lambda)\\
&+\Big(\int_{S^2_+}(\beta_2,\beta_3)\otimes
(\beta_2,\beta_3)\,\,d\nu_{t,x}^\infty(\beta) \Big) m,\\
{\rm Tr}(\overline{e})(t,x)=&\int_{\R_+\times\R^2}
(-\tilde{f}+Id)^{-1}
(x,(\lambda_2,\lambda_3))\cdot(-\tilde{f}+Id)^{-1}
(x,(\lambda_2,\lambda_3))\,\,d\mu_{t,x}(\lambda)\\
&+\Big(\int_{S^2_+}\beta_2^2+\beta_3^2 \,\,d\nu_{t,x}^\infty(\beta) \Big) m,\\
\overline{f}(t,x)=&\int_{\R_+\times\R^2} \lambda_1
\tilde{f}\circ (-\tilde{f}+Id)^{-1}
(x,(\lambda_2,\lambda_3))d\mu_{t,x}(\lambda),\\
\chi(t,x)=&\int_{\R_+\times\R^2} \lambda_1 \tilde {f}\circ (-\tilde{f}+Id)^{-1}
(x,(\lambda_2,\lambda_3))\cdot (-\tilde{f}+Id)^{-1}
(x,(\lambda_2,\lambda_3))d\mu_{t,x}(\lambda)
\end{split}
\end{equation*}
for almost all $(t,x)\in[0,T]\times\T^2$.

Define for every $\mu$-measurable set $A\times B\subset\R^+\times\R^2$ the push-forward of $\mu$ through the map $(-\tilde f+Id)$ as
\begin{equation*}
\nu(A\times B):=(-\tilde f+Id)_\#\mu(A\times B)=\mu(A\times (-\tilde f+Id)^{-1}(B)).
\end{equation*}


Hence using again the variables $(\lambda_1,\lambda')\in\R^+\times\R^2$  (which correspond to $\lambda_1,\lambda_2, \lambda_3$) and $(\beta_1,\beta')\in\mathbb{S}^+$  (corresponding to $\beta_1, \beta_2, \beta_3$), when integrating with respect to $\nu_{t,x}$ and $\nu_{t,x}^\infty$, respectively, the problem can be translated to
\begin{equation*}
\begin{aligned}
\bar{h}&= \langle\lambda_1,\nu\rangle\\
\overline{h^2}&=\langle\lambda_1^2,\nu\rangle+\langle\beta_1^2,\nu^\infty\rangle m\\
\overline{hu}&=\langle\sqrt{\lambda_1}\lambda',\nu\rangle\\
\overline{hu\otimes u}&=\langle\lambda'\otimes\lambda',\nu\rangle+\langle\beta'\otimes\beta',\nu^\infty\rangle m\\
\overline{h|u|^2}&=\langle|\lambda'|^2,\nu\rangle+\langle|\beta'|^2,\nu^\infty\rangle m\\
\overline{h(-dB(u)+f)}&=
\langle \lambda_1
\tilde{f}\circ (-\tilde{f}+Id)^{-1}
(x,\lambda'),\mu\rangle.
%
\end{aligned}
\end{equation*}
We say that $(\mu,m,\mu^\infty)$ is a  \term{measure-valued solution} of~\eqref{savhut} with initial data $(h_0,u_0)$ if
 for every $\tau\in[0,T]$, $\psi\in C^1([0,T]\times\T^2;\R)$, $\phi\in C^1([0,T]\times\T^2;\R^2)$ it holds that
\begin{equation}\label{mass_momentum}
\begin{aligned}
\int_0^\tau\int_{\T^2}\partial_t\psi \bar{h}+\nabla\psi\cdot\overline{hu}dxdt+\int_{\T^2}\psi(x,0)h_0-\psi(x,\tau)\bar{h}(x,\tau)dx&=0,\\
\int_0^\tau\int_{\T^2}\partial_t\phi\cdot\overline{hu}+\nabla\phi : \overline{hu\otimes u}+a\diverg\phi\overline{h^2}
+\phi\cdot\overline{h(-dB(u)+f)}&dxdt\\
+\int_{\T^2}\phi(x,0)\cdot h_0u_0-\phi(x,\tau)\cdot\overline{hu}(x,\tau)&=0.
\end{aligned}
\end{equation}

For a.e.\ $t$, we set
\begin{equation*}
E_{mvs}(t):=\int_{\T^2}\frac{1}{2}\overline{h|u|^2}(t,x)+a\overline{h^2}(t,x)dx
\end{equation*}
and
\begin{equation*}
E_0:=\int_{\T^n}\frac{1}{2}h_0|u_0|^2(x)+h_0^2(x)dx.
\end{equation*}
We say that a measure-valued solution is \term{admissible} 
if
\begin{equation}\label{Emvsenergy}
\begin{aligned}
E_{mvs}(t)\leq E_0-\int_0^t\int_{\T^n}d\overline{h(B(u)-f)\cdot u}(t,x)
\end{aligned}
\end{equation}
in the sense of distributions.


\subsection{Weak-Strong Uniqueness}
\begin{theorem}\label{weak-strong}
Let $f\in L^{\infty}([0,T];L^2(\T^2))$ and suppose $H\in W^{1,\infty}([0,T]\times\T^2), U\in C^1([0,T]\times\T^2)$ is a solution of~\eqref{savhut} with initial data $h_0\geq c>0$, $h_0\in L^2(\T^2)$, $h_0|u_0|^2\in L^1(\T^2)$ and $H(t,x)\geq c>0$ for some constant $c$ and every $(t,x)\in[0,T]\times\T^n$. If $(\nu,m,\nu^\infty)$ is an admissible measure-valued solution with the same initial data, then 
\begin{equation*}
\nu_{t,x}=\delta_{(H(t,x),\sqrt{H(t,x)}U(t,x))} \text{ for a.e. $t,x$, and $m=0$.}
\end{equation*} 
\end{theorem}
\begin{proof}
Let us first define for a.e.\ $t\in[0,T]$ the \term{relative energy} between $(H,U)$ and the measure-valued solution as
\begin{equation*}
\begin{aligned}
E_{rel}(t)&=\int_{\T^2}\frac{1}{2}\overline{h|u-U|^2}+a\overline{(h-H)^2}dx\\
&=\int_{\T^2}\frac{1}{2}\langle\left|\lambda'-\sqrt{\lambda_1}U\right|^2,\nu_{t,x}\rangle dx+\frac{1}{2}\int_{\T^2}\langle|\beta'|^2,\nu_{t,x}^\infty\rangle dm_t(x)\\
&\hspace{1cm}+a\int_{\T^2}\langle|\lambda_1-H|^2,\nu_{t,x}\rangle dx+a\int_{\T^2}\langle \beta_1^2,\nu^\infty_{t,x}\rangle dm_t(x).
\end{aligned}
\end{equation*}
Here, the measure $m_t\in\mathcal{M}^+(\T^2)$ is obtained by the disintegration $m(dtdx)=m_t(dx)\otimes dt$, which is well-defined thanks to the admissibility.

It is straightforward to observe that $E_{rel}(t)=0$ for a.e.\ $t$ implies Theorem~\ref{weak-strong}. 

Following the computations of Section~\ref{Euler} we arrive at
\begin{equation*}
E_{rel}(\tau)+\int_0^\tau\int_{\T^2}\overline{h(dB(u)-dB(U))\cdot (u-U)}dxdt\leq C\norm{U}_{C^1}\int_0^\tau E_{rel}(t)dt.
\end{equation*}
Finally, since $B$ is monotone, the integral on the left hand side is non-negative, and from Gronwall's inequality it follows that $E_{rel}(\tau)=0$ for a.e. $t$.
\end{proof}

\section{Dissipation of Momentum in Finite Time}
\begin{theorem}\label{mvsfinitetime}
Let $(\nu,m,\nu^\infty)$ be an admissible measure-valued solution of the Savage-Hutter system~\eqref{savhut} with initial energy $E_0$ and 
\begin{equation*}
\norm{f}_{L^\infty(\R^+\times\T^2)}<d.
\end{equation*}
Then there exists $0\leq T<\infty$ such that 
\begin{equation*}
M(t):=\int_{\T^2}\overline{h|u|}(t,x)dx=0\hspace{0.3cm}\text{for almost every $t>T$.}
\end{equation*}
Moreover, there exists a constant $C$ depending only on $d-\norm{f}_\infty$ and $a$ such that
\begin{equation*}
T\leq CE_0^{1/4}.
\end{equation*}
\end{theorem}

\begin{proof}
For the momentum we have
\begin{equation*}
M(t)=\int_{\T^2}\overline{h|u|}(t,x)dx=\int_{\T^2}\langle\sqrt{\lambda_1}|\lambda'|,\nu_{t,x}\rangle dx.
\end{equation*}
Note in particular that the momentum does not concentrate, i.e.\ the 2-2-recession function of $(\lambda_1,\lambda')\mapsto \lambda_1|\lambda'|$ is zero.

For the following estimate, we use Jensen's inequality applied to the function $|\cdot|^{4/3}$ (recall that, according to our convention, the torus has measure 1), then Young's inequality,
\begin{equation*}
ab\leq\frac{a^p}{p}+\frac{b^q}{q},\hspace{0.3cm}\frac{1}{p}+\frac{1}{q}=1,
\end{equation*}
 with the conjugate exponents $3$ and $3/2$, and finally the admissibility assumption:
\begin{equation*}
\begin{aligned}
M(t)^{4/3}&=\left(\int_{\T^2}\langle\sqrt{\lambda_1}|\lambda'|,\nu_{t,x}\rangle dx\right)^{4/3}\\
&\leq \int_{\T^2}\langle\sqrt{\lambda_1}^{4/3}|\lambda'|^{4/3},\nu_{t,x}\rangle dx\\
&\leq \frac{1}{3}\int_{\T^2}\langle\lambda_1^2,\nu_{t,x}\rangle dx+\frac{2}{3}\int_{\T^2}\langle|\lambda'|^{2},\nu_{t,x}\rangle dx\\
&\leq C(a) E_{mvs}(t)\leq C(a)\left(E_0-\int_0^t\int_{\T^2}(d-\norm{f}_\infty)\langle\sqrt{\lambda_1}|\lambda'|,\nu_{s,x}\rangle dxds\right)\\
&= C(a)\left(E_0-(d-\norm{f}_\infty)\int_0^tM(s)ds\right),
\end{aligned}
\end{equation*}
where 
\begin{equation*}
C(a)=\max\left\{\frac{1}{3a},\frac{4}{3}\right\}.
\end{equation*}
Therefore, for almost every $t$, $M(t)$ is less than or equal to the solution of the integral equation
\begin{equation*}
\tilde{M}(t)^{4/3}=C(a)E_0-C(a)(d-\norm{f}_\infty)\int_0^t\tilde{M}(s)ds
\end{equation*}
or equivalently (after differentiating)
\begin{equation*}
\tilde{M}'(t)=-\frac{3}{4}C(a)(d-\norm{f}_\infty)\tilde{M}(t)^{2/3},\hspace{0.2cm}\tilde{M}(0)=(C(a)E_0)^{3/4}.
\end{equation*}
The solution of this ordinary differential equation is easily computed as
\begin{equation*}
M(t)=\begin{cases}\left[\frac{1}{3}(3(C(a)E_0)^{1/4}-\frac{3}{4}C(a)(d-\norm{f}_\infty)t)\right]^3 & \text{if $3(C(a)E_0)^{1/4}-\frac{3}{4}C(a)(d-\norm{f}_\infty)t\geq0$,}\\
0 & \text{otherwise.}
\end{cases}
\end{equation*}
In fact,  $\left[\frac{1}{3}(3(C(a)E_0)^{1/4}-\frac{3}{4}C(a)(d-\norm{f}_\infty)t)\right]^3$ would also be a solution for all times, but we know a priori that $M(t)$ must be non-negative.

It follows that there is a time 
\begin{equation*}
T\leq\frac{4}{d-\norm{f}_\infty}C(a)^{-3/4}E_0^{1/4}
\end{equation*}
 after which $M(t)=0$ almost everywhere.

\end{proof}

\begin{corollary}
Let $(h,u)$ be an admissible weak solution of the Savage-Hutter equations with initial energy $E_0$ and $\norm{f}_\infty<d$. Then there exists a time $0\leq T<\infty$ such that for almost every $t>T$, $(h,u)$ is stationary, i.e.\ $u(t,x)=0$ for almost every $t>T$ and $x\in\T^2$, $\partial_th(t,x)=0$ for almost every $t>T$, $x\in\T^2$, and
\begin{equation*}
\left|\nabla h-\frac{f}{2a}\right|\leq\frac{d}{2a}.
\end{equation*}
Moreover, $T$ satisfies the estimate of Theorem~\ref{mvsfinitetime}.
\end{corollary}
\begin{proof}
As every admissible weak solution can be viewed as an admissible measure-valued solution via the identification $\nu=\delta_{(h,\sqrt{h}u)}$, $m=0$, from Theorem~\ref{mvsfinitetime} we obtain a time $T$ such that after this time, the momentum is zero:
\begin{equation*}
\int_{\T^2}h|u|dx=0\hspace{0.2cm}\text{for almost every $t>T$.}
\end{equation*}
Therefore, the Savage-Hutter equations reduce to $\partial_th=0$ and
\begin{equation*}
\nabla(ah^2)=-dhB(u)+hf.
\end{equation*}
The latter is clearly equivalent to $|\nabla h-f/2a|\leq d/2a$, given that $|B(u)|\leq1$.
\end{proof}

\bibliography{Thesis}
\bibliographystyle{amsalpha}


\end{document}